\documentclass[12pt,reqno]{amsart}
\usepackage{amssymb}
\usepackage{subfiles}
\usepackage{stackengine}
\usepackage{mathtools}
\usepackage[utf8]{inputenc}
\usepackage{color}
\usepackage{tikz}
\usepackage[a4paper,top=3cm,bottom=3cm,inner=3cm,outer=3cm]{geometry}
\usepackage{mathrsfs,amsmath,amsfonts,amsthm,graphicx,pdfpages}
\usepackage{comment}
\usepackage{graphicx}
\usepackage[all,2cell]{xy}\UseAllTwocells\SilentMatrices
\definecolor{darkgreen}{rgb}{0,0.45,0}
\usepackage[colorlinks,citecolor=darkgreen,urlcolor=gray,final,hyperindex,linktoc=page]{hyperref}
\usepackage[capitalize]{cleveref}
\crefname{equation}{}{}
\crefname{item}{}{}
\usepackage{enumerate}

\definecolor{purple(x11)}{rgb}{0.5, 0.0, 0.5}

\usetikzlibrary{decorations.markings,arrows.meta,calc,fit,quotes,cd,math,external}
  
\tikzset{tick/.style={postaction={decorate,decoration={markings,
mark=at position 0.5 with {\draw[-] (0,.4ex) -- (0,-.4ex);}}}}}
\tikzset{tickx/.style={postaction={decorate,decoration={markings,mark=at position 0.5 with
{\fill circle [radius=.28ex];}}}}}
\newcommand{\tickar}
{\begin{tikzcd}[baseline=-0.5ex,cramped,sep=small,ampersand replacement=\&]{}\ar[r,tick]\&{}\end{tikzcd}}

\newtheorem*{thm*}{Theorem}
\theoremstyle{remark}
\newtheorem*{rmk*}{Remark}
\newtheorem*{lem*}{Lemma}
\theoremstyle{definition}
\newtheorem*{defi*}{Definition}
\newtheorem*{cor*}{Corollary}
\theoremstyle{definition}
\newtheorem*{examples*}{Examples}
\newtheorem{prop*}{Proposition}

\theoremstyle{plain}
\newtheorem{thm}{Theorem}[section]
\theoremstyle{plain}
\newtheorem{prop}[thm]{Proposition}
\theoremstyle{remark}
\newtheorem{rmk}[thm]{Remark}
\theoremstyle{plain}
\newtheorem{lem}[thm]{Lemma}
\theoremstyle{plain}
\newtheorem{cor}[thm]{Corollary}
\theoremstyle{definition}
\newtheorem{defi}[thm]{Definition}
\theoremstyle{definition}
\newtheorem{examples}[thm]{Example}

\DeclareFontFamily{U}{mathx}{\hyphenchar\font45}
\DeclareFontShape{U}{mathx}{m}{n}{
      <5> <6> <7> <8> <9> <10>
      <10.95> <12> <14.4> <17.28> <20.74> <24.88>
      mathx10
      }{}
\DeclareSymbolFont{mathx}{U}{mathx}{m}{n}
\DeclareFontSubstitution{U}{mathx}{m}{n}
\DeclareMathAccent{\widecheck}{0}{mathx}{"71}
\newcommand{\wt}{\widetilde}
\newcommand{\padj}{\;\stackunder[-1pt]{$\dashv$}{$\scriptscriptstyle p\;$}\;} 
\newcommand{\WCat}{\caa{W}\textrm{-}\Cat}

\newcommand{\VMat}{\ca{V}\textrm{-}\B{Mat}}

\newcommand{\ca}{\mathcal}
\newcommand{\caa}{\mathbb} 
\newcommand{\Hom}{\ensuremath{\mathrm{Hom}}}

\newcommand{\Comod}{\ensuremath{\mathbf{Comod}}}
\newcommand{\Mod}{\ensuremath{\mathbf{Mod}}}
\newcommand{\Cat}{\B{Cat}}

\newcommand{\Alg}{\ensuremath{\mathbf{Alg}}}
\newcommand{\Coalg}{\ensuremath{\mathbf{Coalg}}}
\newcommand{\Mon}{\ensuremath{\mathbf{Mon}}}
\newcommand{\MonCat}{\mathbf{MonCat}}
\newcommand{\Comon}{\ensuremath{\mathbf{Comon}}}

\newcommand{\ob}{\ensuremath{\mathrm{ob}}}

\newcommand{\opl}{\mathrm{opl}}
\newcommand{\lax}{\mathrm{lax}}
\newcommand{\Cart}{\ensuremath{\mathrm{Cart}}}
\newcommand{\Cocart}{\ensuremath{\mathrm{Cocart}}}

\newcommand{\B}{\mathbf}

\newcommand{\pullbackcorner}[1][dr]{\save*!/#1+1.2pc/#1:(1,-1)@^{|-}\restore}

\newcommand{\op}{\mathrm{op}}

\newcommand{\sbul}{\scriptstyle\bullet}
\newcommand{\tick}{\object@{|}}

\newcommand{\VCocat}{\ca{V}\textrm{-}\B{Cocat}}
\newcommand{\VCat}{\ca{V}\textrm{-}\B{Cat}}
\newcommand{\VMod}{\ca{V}\textrm{-}\B{Mod}}
\newcommand{\VComod}{\ca{V}\textrm{-}\B{Comod}}
\newcommand{\VGrph}{\ca{V}\textrm{-}\B{Grph}}

\newcommand{\Set}{\B{Set}}
\newcommand{\matr}[3]{\SelectTips{eu}{10}\xymatrix@C=.2in{#1\colon #2\ar[r]|-{\object@{|}} & #3}}
\newcommand{\proar}[3]{\SelectTips{eu}{10}\xymatrix@C=.2in{#1\colon #2\ar[r]|-{\sbul} & #3}}

\newcommand{\Fib}{\B{Fib}}
\newcommand{\OpFib}{\B{OpFib}}

\begin{document}
\title{On Enriched Fibrations*}

\author{Christina Vasilakopoulou}
\thanks{*This is a provisory draft of a paper whose definitive version is due to be published in the
``Cahiers de Topologie et Géométrie Différentielle Catégoriques".}
\address{Department of Mathematics, University of California, Riverside, 900 University Avenue, 92521, USA}
\email{vasilak@ucr.edu}

\begin{abstract}
We introduce the notion of an enriched fibration, i.e. a fibration whose total category and base
category are enriched in those of a monoidal fibration in an appropriate way. Furthermore, we provide
a way to obtain such a structure, starting from actions of monoidal categories
with parameterized adjoints. The motivating goal is to capture certain example cases,
like the fibration of modules over algebras enriched in the opfibration of comodules over
coalgebras.
\end{abstract}

\maketitle

\setcounter{tocdepth}{1}
\tableofcontents

\section{Introduction}

Enriched category theory \cite{Kelly}, as well as the theory of fibrations \cite{Grothendieckcategoriesfibrees},
have both been of central importance to developments in many contexts.
Both are classical theories for formal category theory; however,
they do not seem to `go together' in some evident way.

The goal of the present work is to introduce a notion of an \emph{enriched fibration}. This should combine
elements of both concepts in an appropriate and natural way; the enriched structure of a category
cannot really be internalized in order to provide a definite answer.
In any case, `being enriched in' and `being internal to' are two major but separate generalizations
of ordinary category theory, whereas `being fibred over' is often considered as a third one.

More explicitly, we would like to characterize a fibration as being enriched in
some special kind of fibration, serving similar purposes as the monoidal base of usual
enriched categories; this has already been identified as a \emph{monoidal fibration}
\cite{Framedbicats}. For the desired enriched fibration definition, there are two main factors that determine its relevance.
First of all, it should be able to adequately capture certain cases
that first arose in \cite{PhDChristina} and furthermore studied in \cite{Measuringcomonoid,Measuringcomodule,VCocats},
and in fact motivated these explorations. Further details of these examples and how
they ultimately fit in the described framework can be found in \cref{sec:applications}. The original driving example case
is the enrichment of algebras in coalgebras via Sweedler's \emph{measuring coalgebra} construction \cite{Sweedler},
together with the enrichment of a global category of modules in comodules; the latter categories are respectively
fibred and opfibred over algebras and coalgebras. This also extends to their many-object generalizations, namely
(enriched) categories and cocategories and their (enriched) modules and comodules. These cases can be roughly depicted as
\begin{displaymath}
 \xymatrix @C=.6in @R=.4in
{\Mod\ar@{-->}[r]^-{\mathrm{enriched}}
\ar[d]_-{\mathrm{fibred}} & \Comod
\ar[d]^-{\mathrm{opfibred}} \\
\Alg(\ca{V})\ar@{-->}[r]_-{\mathrm{enriched}} &
\Coalg(\ca{V})}\qquad\qquad
 \xymatrix @C=.6in @R=.4in
{\ca{V}\text{-}\Mod\ar@{-->}[r]^-{\mathrm{enriched}}
\ar[d]_-{\mathrm{fibred}} & \ca{V}\text{-}\Comod
\ar[d]^-{\mathrm{opfibred}} \\
\ca{V}\textrm{-}\B{Cat}\ar@{-->}[r]_-{\mathrm{enriched}} &
\ca{V}\textrm{-}\B{Cocat}.}
\end{displaymath}

Secondly, the introduced enriched fibration concept should theoretically constitute an as-close-as-possible fibred analogue
of the usual enrichment of categories. In order to initiate such an effort, we provide a theorem which
under certain assumptions ensures the existence of such a structure. This theorem lifts
a standard result, which combines the theory of \emph{actions} of monoidal categories and parameterized adjunctions
to produce an enrichment \cite{enrthrvar,AnoteonActions}, to the fibred context.

Notably, a strongly related notion called \emph{enriched indexed category} has been studied, from a slightly different point of view,
originally in \cite{Bunge} and also independently in \cite{Enrichedindexedcats}. However, the
main definitions and constructions diverge from the ones presented here. We postpone a short discussion
on these differences until the very end of the paper, \cref{sec:discussion}.

Finally, it should be indicated that this paper deliberately includes only what is necessary to first of all
sufficiently describe the examples at hand.
It elaborates on and extends a sketched narrative from \cite[\S 8.1]{PhDChristina},
and provides the first steps in such a research direction. Future work may build on the current development,
towards a theory of enriched fibrations and related structures.

\subsection*{Acknowledgements}
The author would like to thank Martin Hyland for posing the driving questions that led to this work, as well as providing valuable insight; also Marta Bunge
for carefully reviewing this manuscript and offering multiple useful comments; finally Mitchell Buckley, for many discussions that affected
various parts of this work.

\section{Background}

In this section, we recall some basic definitions and known results which serve
as background material in what follows, and we also fix terminology.

\subsection{Monoidal categories, actions and enrichment}

We assume familiarity with the basics of monoidal categories, see \cite{BraidedTensorCats,MacLane}.
A monoidal category is denoted by $(\ca{V},\otimes,I)$ with associator and left and right
unit constraints $a,\ell,r$. A lax monoidal structure on a functor $F\colon\ca{V}\to\ca{W}$ between monoidal
categories is denoted by $(\phi,\phi_0)$, with components $\phi_{AB}\colon FX\otimes FY\to F(X\otimes Y)$
and $\phi_0\colon I\to FI$ satisfying usual axioms.
If these are isomorphisms/identities, this is a strong/strict monoidal structure.

A (left) monoidal closed category is one where the functor $(-\otimes X)$ has a right adjoint $[X,-]$,
for all objects $X$. This induces the internal hom functor $[-,-]\colon\ca{V}^\op\times\ca{V}\to\ca{V}$,
as a result of the classic \emph{parameterized adjunctions} theorem \cite[\S IV.7.3]{MacLane}:

\begin{thm}\label{thm:parameterCat}
Suppose that, for a functor of two variables $F:\ca{A}\times\ca{B}\to\ca{C}$,
there exists an adjunction
\begin{displaymath}
 \xymatrix @C=.6in
{\ca{A}\ar@<+.8ex>[r]^-{F(-,B)}\ar@{}[r]|-\bot &
\ca{C}\ar@<+.8ex>[l]^-{G(B,-)}}
\end{displaymath}
for each object $B\in\ca{B}$, with an isomorphism
$\ca{C}(F(A,B),C)\cong\ca{A}(A,G(B,C))$
natural in $A$ and $C$. Then, there is a unique way to make 
$G$ into a functor of two variables
$\ca{B}^\mathrm{op}\times\ca{C}\to\ca{A}$
for which the isomorphism is natural also in $B$.
\end{thm}

The functor $G$ is called the \emph{(right) parameterized adjoint} of $F$, and we denote this as $F\padj G$. In particular,
$\otimes\padj [-,-]$ in any monoidal (left) closed category. We could also decide to fix the other
parameter, and have that $F(A,-)\dashv H(A,-)$ for $H\colon\ca{A}^\op\times\ca{C}\to\ca{B}$. For a 2-categorical proof and generalizations,
see \cite{Multivariableadjunctions}.

We now recall some basics of the theory of actions of monoidal categories, \cite{AnoteonActions}.

\begin{defi}\label{def:action}
A \emph{(left) action} of a monoidal category $\ca{V}$ on a category $\ca{D}$ is given by a functor
$*\colon\ca{V}\times\ca{D}\to\ca{D}$ along with two natural isomorphisms $\chi,\nu$ with components
\begin{equation}\label{eq:chinu}
 \chi_{XYD}\colon (X\otimes Y)* D\xrightarrow{\;\sim\;}X*(Y* D),\qquad
 \nu_{D}\colon I* D\xrightarrow{\;\sim\;}D
\end{equation}
satisfying the commutativity of
\begin{equation}\label{actiondiag} 
\xymatrix{((X\otimes Y)\otimes
Z)*D\ar[r]^-{\chi}\ar[d]_-{a*1} & (X\otimes Y)*(Z*D)\ar[r]^-{\chi}
& X*(Y*(Z*D))\\ (X\otimes(Y\otimes Z))*D\ar[rr]_-{\chi} &&
X*((Y\otimes Z)*D)\ar[u]_-{1*\chi}}
\end{equation}
\begin{displaymath}
\xymatrix@R=.3in@C=.15in{(I\otimes X)*D\ar[rr]^-{\chi}\ar[dr]_-{l*1} 
&& I*(X*D)\ar[dl]^-{\nu}\\ &
X*D &} \;\; 
\xymatrix@R=.3in@C=.15in{(X\otimes I)*D\ar[rr]^-{\chi}\ar[dr]_-{r*1} 
&& X*(I*D)\ar[dl]^-{1*{\nu}}\\
& X*D &}
\end{displaymath}
The category $\ca{D}$ is called a $\ca{V}$-\emph{representation}, or a $\ca{V}$-\emph{actegory} \cite{Mccruddencoalgebroidsreps}.
\end{defi}

For example, every monoidal category has a canonical action on itself via its
tensor product, $\otimes=*\colon\ca{V}\times\ca{V}\to\ca{V}$, and $\chi=a,\nu=\ell$; it is called the \emph{regular} $\ca{V}$-representation.
Moreover, for any monoidal closed category, its internal hom constitutes an action of the monoidal
$\ca{V}^\op$ (with the same tensor product $\otimes^\op$) on $\ca{V}$, via the standard natural isomorphisms 
\begin{displaymath}
\chi_{XYZ}:[X\otimes Y,D]\xrightarrow{\sim}[X,[Y,Z]], \quad \nu_D:[I,D]\xrightarrow{\sim}D
\end{displaymath}
which satisfy \cref{actiondiag} using the transpose diagrams under the tensor-hom adjunction.

Familiarity with enrichment theory is also assumed, see \cite{Kelly}.
We denote the 2-category of $\ca{V}$-enriched categories, along with enriched functors and enriched natural transformations, $\VCat$;
we call $\ca{V}$ the monoidal base of the enrichment.
If $\ca{A}$ is a $\ca{V}$-enriched category
with hom-objects $\ca{A}(A,B)\in\ca{V}$, we will write $j_A\colon I\to \ca{A}(A,A)$ for its identites
and $M_{ABC}\colon\ca{A}(B,C)\otimes\ca{A}(A,B)\to\ca{A}(B,C)$ for the composition. 
Its \emph{underlying category} $\ca{A}_0$ has the same
objects, while morphisms $f\colon A\to B$ in $\ca{A}_0$ are just `elements' $f\colon I\to\ca{A}(A,B)$ in $\ca{V}$,
i.e. $\ca{A}_0(A,B)=\ca{V}(I,\ca{A}(A,B))$ as sets.
In fact, we can define a functor
\begin{equation}\label{enrichedhomfunctor}
 \ca{A}(-,-)\colon \ca{A}_0^\op\times\ca{A}_0\to\ca{V}
\end{equation}
called the \emph{enriched hom-functor},
which maps $(A,B)$ to $\ca{A}(A,B)$, and a pair of arrows 
$(A'\xrightarrow{f}A,B\xrightarrow{g}B')$
in $\ca{A}_0^\mathrm{op}\times\ca{A}_0$ to the 
top arrow
\begin{displaymath}
 \xymatrix@R=.3in
 {\ca{A}(A,B)\ar@{-->}[rr]^-{\ca{A}(f,g)} 
 \ar[d]_-{r^{-1}} && \ca{A}(A',B') \\
 \ca{A}(A,B)\otimes I\ar[d]_-{1\otimes f} && 
 \ca{A}(B,B')\otimes\ca{A}(A',B)\ar[u]_-M \\\
 \ca{A}(A,B)\otimes\ca{A}(A',A)\ar[r]_-M &
 \ca{A}(A',B)\ar[r]_-{l^{-1}} & 
 I\otimes\ca{A}(A',B)\ar[u]_-{g\otimes 1}}
\end{displaymath}

Speaking loosely, we say that an ordinary category $\ca{C}$
is enriched in a monoidal category $\ca{V}$ when we have a 
$\ca{V}$-enriched category $\ca{A}$ (often denoted \underline{$\ca{C}$}) and an isomorphism
$\ca{A}_0\cong\ca{C}$. Consequently, \emph{to be enriched 
in} $\ca{V}$ is not a property, 
but additional structure. Of course, 
a given ordinary category may be enriched 
in more than one 
monoidal category; this is evident in view of \cref{changeofbase}.
But also, a category  $\ca{C}$ may be enriched in $\ca{V}$ in more than one way.

\begin{prop}[Change of Base]\label{changeofbase}
Suppose $F:\ca{V}\to\ca{W}$ is a lax
monoidal functor between two monoidal categories. There 
is an induced 2-functor 
\begin{displaymath}
\wt{F}\colon\ca{V}\textrm{-}\B{Cat}\longrightarrow\ca{W}\textrm{-}\B{Cat}
\end{displaymath}
between the 2-categories of $\ca{V}$ and $\ca{W}$-enriched categories, 
which maps any $\ca{V}$-category $\ca{A}$ to a 
$\ca{W}$-category with the same objects as $\ca{A}$ and hom-objects 
$F\ca{A}(A,B)$.
\end{prop}

\begin{proof}[Sketch of proof]
On the level of objects, the composition and identities are given by
\begin{displaymath}
\xymatrix @C=.7in @R=.5in
{F\ca{A}(B,C)\otimes F\ca{A}(A,B)\ar[d]_-{\phi_{\ca{A}(B,C),\ca{A}(A,B)}}
\ar @{-->}[r] & F\ca{A}(A,C)\\
F(\ca{A}(B,C)\otimes\ca{A}(A,B))\ar[ru]_-{FM_{ABC}}}\qquad
\xymatrix @C=.7in @R=.5in
{I_\ca{W}\ar[d]_-{\phi_0}\ar @{-->}[r] & F\ca{A}(A,A)\\
FI_\ca{V}\ar[ur]_-{Fj_A} &}
\end{displaymath}
\end{proof}

A crucial result for what follows is that given a category $\ca{D}$ with an action
from a monoidal category $\ca{V}$ with a parameterized
adjoint, we obtain a $\ca{V}$-enriched category.

\begin{thm}\label{actionenrich}
Suppose that $\ca{V}$ is a monoidal category which acts on a category $\ca{D}$ via a functor 
$*:\ca{V}\times\ca{D}\to\ca{D}$, such that $-*D$ has a right adjoint $F(D,-)$ for every $D\in\ca{D}$.
Then we can enrich $\ca{D}$ in $\ca{V}$, with hom-objects $\underline{\ca{D}}(A,B)=F(A,B)$.
\end{thm}

The proof and further details can be found in \cite{AnoteonActions} or \cite[\S~4.3]{PhDChristina}.
Briefly, due to the adjunction $-*D\dashv F(D,-)$, we have natural isomorphisms
\begin{equation}\label{eq:natiso}
\ca{D}(X*D,E)\cong\ca{V}(X,F(D,E))
\end{equation}
which give rise to a functor $F\colon\ca{D}^\op\times\ca{D}\to\ca{V}$ by \cref{thm:parameterCat}. This serves
as the enriched hom-functor of the induced enrichment of $\ca{D}$ in $\ca{V}$: we can define a composition law
$F(B,C)\otimes F(A,B)\to F(A,C)$ as the adjunct of the composite
\begin{equation}\label{compositionunderadjunction}
\left(F(B,C)\otimes F(A,B)\right)*A\xrightarrow{\chi}F(B,C)*\left(F(A,B)*A\right)\xrightarrow{1*\varepsilon_B}F(B,C)*B\xrightarrow{\varepsilon_C}C 
\end{equation}
and identities $I\to F(A,A)$ as the adjuncts of
\begin{equation}\label{identityunderadjunction}
I*A\xrightarrow{\nu}A
\end{equation}
where $\chi$ and $\nu$ give the action structure \cref{eq:chinu}
and $\varepsilon$ is the counit of the adjunction. The associativity and identity axioms for the enrichment can be established using
the action axioms. Finally, $\underline{\ca{D}}_0\cong\ca{D}$ since they have the same objects, and
\begin{displaymath}
 \underline{\ca{D}}_0(A,B)=\ca{V}(I,F(A,B))\stackrel{\cref{eq:natiso}}{\cong}\ca{D}(I*A,B)\stackrel{\nu}{\cong}\ca{D}(A,B).
\end{displaymath}
In fact, \cref{actionenrich} gives part of one direction of an equivalence
\begin{displaymath}
 \ca{V}\text{-}\mathrm{Rep}_{\textrm{cl}}\simeq
\ca{V}\text{-}\B{Cat}_{\otimes}
\end{displaymath}
between \emph{closed} $\ca{V}$-representations (those with action equipped with a pa\-ra\-me\-te\-ri\-zed
adjoint) and \emph{tensored} $\ca{V}$-categories (those with specific weighted limits), for $\ca{V}$ a monoidal closed
category. This equivalence, discussed in \cite[\S 6]{AnoteonActions}, is a special case of the
much more general \cite[Theorem 3.7]{enrthrvar} for bicategory-enriched categories.

\begin{rmk}\label{rmk:VenrichedinV}
When $\ca{V}$ is monoidal closed,
the regular action $\otimes\colon\ca{V}\times\ca{V}\to\ca{V}$ has a parameterized adjoint $[-,-]\colon\ca{V}^\op\times\ca{V}\to\ca{V}$.
We thus re-discover the well-known enrichment of a monoidal closed category in itself via the internal hom,
as a direct application of \cref{actionenrich}.
\end{rmk}

\subsection{Pseudomonoids and pseudomodules}

Recall that a \emph{monoidal 2-category} $(\ca{K},\otimes,I)$ is a 2-category equipped with a pseudofunctor
$\otimes\colon\ca{K}\times\ca{K}\to\ca{K}$ and a unit $I\colon\B{1}\to\ca{K}$ which are associative and unital up to coherence
equivalence, see \cite{CoherenceTricats}.

\begin{defi}\cite[\S 3]{Monoidalbicatshopfalgebroids}\label{def:pseudomon}
A \emph{pseudomonoid} $A$ in $\ca{K}$ is an object equipped with multiplication $m\colon A\otimes A\to A$
and unit $j\colon I\to A$ along with invertible 2-cells
satisfying coherence conditions.
\begin{equation}\label{pseudomon2cells}
\begin{tikzcd}
A\otimes A\otimes A\ar[r,"1\otimes m"]\ar[d,"m\otimes1"']\ar[dr,phantom,"\scriptstyle\stackrel{a}{\cong}"description]
& A\otimes A\ar[d,"m"]
& A\otimes I\ar[r,"1\otimes j"]\ar[dr,"\sim"'] & A\otimes A\ar[d,"m","\stackrel{\ell}{\cong}\quad"'near start,
"\quad\;\stackrel{r}{\cong}"near start] &
I\otimes A\ar[l,"j\otimes1"']\ar[dl,"\sim"] \\
A\otimes A\ar[r,"m"'] & A && A &
\end{tikzcd}
\end{equation}
\end{defi}

The notion of a pseudomodule for a pseudomonoid in a monoidal 2-category (or bicategory) can be found in similar contexts 
\cite{Marmolejopseudomonads,LackPseudomonads}; conceptually, as it is the case for modules for monoids in a monoidal category, it arises as 
a pseudoalgebra for the pseudomonad $(A\otimes-)$ induced by a pseudomonoid $A$ in $\ca{K}$.

\begin{defi}\label{def:pseudomod}
A (left) $A$-\emph{pseudomodule} is an object $M$ in $(\ca{K},\otimes,I)$ equipped with $\mu\colon A\otimes M\to M$ (the \emph{pseudoaction})
and invertible 2-cells
\begin{equation}\label{pseudomod2cells}
 \begin{tikzcd}
A\otimes A\otimes M\ar[r,"1\otimes\mu"]\ar[d,"m\otimes1"']\ar[dr,phantom,"\scriptstyle\stackrel{\chi}{\cong}"description]
& A\otimes M\ar[d,"\mu"]
& I\otimes M\ar[r,"j\otimes1"]\ar[dr,"\sim"'] & A\otimes M\ar[d,"\mu","\stackrel{\nu}{\cong}\quad"'near start] \\
A\otimes M\ar[r,"\mu"'] & A && A  
\end{tikzcd}
\end{equation}
satisfying coherence conditions.
\end{defi}

\begin{examples}\label{ex:Catpseudo}
As a fundamental example of a (cartesian) monoidal 2-category, consider $\Cat$ equipped with the 2-functor
$\times:\B{Cat}\times\B{Cat}\to\B{Cat}$
and $\ca{I}$ the unit category. It is a standard fact that a pseudomonoid therein is precisely a monoidal category $(\ca{V},\otimes,I,a,\ell,r)$.
Moreover, an action of a monoidal category $\ca{V}$ on an ordinary category $\ca{A}$ as defined in \cref{def:action}
is precisely a $\ca{V}$-pseudoaction inside $(\Cat,\times,\ca{I})$, exhibiting $\ca{A}$ as a $\ca{V}$-pseudomodule. 
\end{examples}

\subsection{Fibrations and adjunctions}
We now briefly recall some basic concepts and constructions from the theory of fibred categories.
Relevant references for what follows are \cite{Handbook2,FibredAdjunctions}.

Consider a functor $P:\ca{A}\to\caa{X}$. 
A morphism $\phi:A\to B$ in $\ca{A}$ over 
a morphism $f=P(\phi):X\to Y$ in $\caa{X}$
is called \emph{cartesian} if and only if, for all 
$g:X'\to X$ in $\caa{X}$ and $\theta:A'\to B$ in 
$\ca{A}$ with 
$P\theta=f\circ g$, there exists a unique arrow $\psi:A'\to A$ 
such that $P\psi=g$ and $\theta=\phi\circ\psi$:
\begin{displaymath}
\xymatrix @R=.1in @C=.6in
{A'\ar [drr]^-{\theta}\ar @{-->}[dr]_-{\exists!\psi} 
\ar @{.>}@/_/[dd] &&& \\
& A\ar[r]_-{\phi} \ar @{.>}@/_/[dd] & 
B \ar @{.>}@/_/[dd] & \textrm{in }\ca{A}\\
X'\ar [drr]^-{f\circ g=P\theta}\ar[dr]_-g &&&\\
& X\ar[r]_-{f=P\phi} & Y & \textrm{in }\caa{X}}\label{eq:13}
\end{displaymath}
For $X\in\ob\caa{X}$, the \emph{fibre
of $P$ over $X$} written $\ca{A}_X$, 
is the subcategory of $\ca{A}$ 
which consists of objects $A$
such that $P(A)=X$ and morphisms $\phi$ with 
$P(\phi)=1_X$, called
\emph{vertical} morphisms. 
The functor $P:\ca{A}\to\caa{X}$ 
is called a \emph{fibration} if and only if, for all $f:X\to Y$ in
$\caa{X}$ and $B\in\ca{A}_Y$, there is a cartesian morphism
$\phi$ with codomain $B$ above $f$; it is called a \emph{cartesian lifting} of $B$ along
$f$. The category $\caa{X}$ is then called the \emph{base} of the fibration,
and $\ca{A}$ its \emph{total} category.

Dually, the functor $U:\ca{C}\to\caa{X}$ is an \emph{opfibration} 
if $U^\mathrm{op}$ is a fibration, \emph{i.e.} for every $C\in\ca{C}_X$
and $g:X\to Y$ in $\caa{X}$, there is a cocartesian morphism
with domain $C$ above $g$, the \emph{cocartesian lifting}
of $C$ along $g$.

If $P:\ca{A}\to\caa{X}$ is a fibration, assuming the 
axiom of choice we may select a 
cartesian arrow over each $f:X\to Y$ in $\caa{X}$
and $B\in\ca{A}_Y$, denoted by 
$\Cart(f,B):f^*(B)\to B$.
Such a choice of cartesian liftings is called a 
\emph{cleavage} for $P$, which is then
called a \emph{cloven} fibration; any fibration is henceforth assumed to be cloven.
Dually, if $U$ is an opfibration, for any $C\in\ca{C}_X$
and $g:X\to Y$ in $\caa{X}$ we can choose a cocartesian
lifting of $C$ along $g$,
$\Cocart(g,C):C\longrightarrow g_!(C)$.
The choice of (co)cartesian liftings in an (op)fibration induces
a so-called \emph{reindexing functor} between the fibre categories
\begin{displaymath}
f^*:\ca{A}_Y\to\ca{A}_X\quad\textrm{ and }\quad g_!\colon\ca{C}_X\to\ca{C}_Y
\end{displaymath}
respectively, for each morphism $f:X\to Y$ and $g\colon X\to Y$ in the base category,
mapping each object to the (co)domain of its lifting.


An \emph{oplax morphism of fibrations} (or oplax fibred 1-cell) $(S,F)$ between $P:\ca{A}\to\caa{X}$ and $Q:\ca{B}\to\caa{Y}$ is
given by a commutative square of categories and functors
\begin{equation}\label{eq:fibred1cell}
\xymatrix @C=.4in @R=.4in
{\ca{A}\ar[r]^-S \ar[d]_-P &
\ca{B}\ar[d]^-Q \\
\caa{X}\ar[r]_-F &
\caa{Y}}
\end{equation}
as in \cite[Def.~3.5]{Framedbicats}. 
If moreover $S$ preserves cartesian arrows, meaning that if $\phi$ is $P$-cartesian then $S\phi$ is $Q$-cartesian,
the pair $(S,F)$ is called a \emph{fibred 1-cell} or \emph{strong morphism of fibrations}.
Dually, we have the notion of an \emph{lax morphism of opfibrations} $(K,F)$,
and \emph{opfibred 1-cell} when $K$ is cocartesian. 
Notice that any oplax fibred 1-cell $(S,F)$ determines a collection of functors between the fibres
$S_X\colon\ca{A}_X\to\ca{B}_{FX}$
as the restriction of $S$ to the corresponding subcategories.

A \emph{fibred 2-cell} between oplax fibred 1-cells $(S,F)$ and $(T,G)$ is a pair of natural transformations 
($\alpha:S\Rightarrow T,\beta:F\Rightarrow G$) with $\alpha$ above $\beta$, \emph{i.e.} $Q(\alpha_A)
=\beta_{PA}$ for all $A\in\ca{A}$, displayed
\begin{equation}\label{eq:fibred2cell}
\xymatrix @C=.8in @R=.5in
{\ca{A}\rtwocell^S_T{\alpha}\ar[d]_-P
& \ca{B}\ar[d]^-Q \\
\caa{X}\rtwocell^F_G{\beta} & \caa{Y}.}
\end{equation}
Notice that if the 1-cells are strong, the definition of a 2-cell between them remains the same.
Dually, we have the notion of an \emph{opfibred 2-cell} 
between (lax) opfibred 1-cells.

We obtain 2-categories $\B{Fib}_\opl$ and $\B{Fib}$ of fibrations over arbitrary base categories, (oplax) fibred 1-cells and
fibred 2-cells. Evidently, these are both subcategories of $\Cat^\B{2}$.
$\Fib_\opl$ is a full sub-2-category of those objects which are fibrations,
and $\Fib$ is the non-full sub-2-category whose morphism are commutative squares
where the top functor is cartesian. Dually, $\OpFib\subset\OpFib_\lax\subset_{\textrm{full}}\Cat^\B{2}$.
These 2-categories are monoidal, inheriting the tensor product from $\B{Cat}^\B{2}$:
the cartesian product of two fibrations is still a fibration. The unit is $1_\ca{I}\colon\ca{I}\to\ca{I}$, the identity
on the terminal category.

Notice that the terminology for oplax morphisms of fibrations and lax morphisms of opfibrations is justified by
a relaxed version of the fundamental equivalence between fibrations and pseudofunctors (Grothendieck construction).
For more details, see \cite[Prop.~3.6]{Framedbicats}.

We now turn to notions of adjunctions between fibrations, internally to any of the above 2-categories
of (op)fibrations.

\begin{defi}\label{generalfibredadjunction}
Given fibrations $P:\ca{A}\to\caa{X}$ and $Q:\ca{B}\to\caa{Y}$,
a \emph{general (oplax) fibred adjunction} $(L,F)\dashv(R,G)$
is given by a pair of (oplax) fibred 1-cells $(L,F):P\to Q$ and 
$(R,G):Q\to P$ together with fibred 2-cells $(\zeta,\eta):(1_\ca{A},1_\caa{X})\Rightarrow
(RL,GF)$ and  $(\xi,\varepsilon):(LR,FG)\Rightarrow(1_\ca{B},1_\caa{Y})$
such that $L\dashv R$ via $\zeta,\xi$ and $F\dashv G$ via $\eta,\varepsilon$. This is displayed as
\begin{displaymath}
\xymatrix @C=.7in @R=.4in
{\ca{A} \ar[d]_-P 
\ar @<+.8ex>[r]^-L\ar@{}[r]|-\bot
& \ca{B} \ar @<+.8ex>[l]^-{R} \ar[d]^-Q \\
\caa{X} \ar @<+.8ex>[r]^-F\ar@{}[r]|-\bot
& \caa{Y} \ar @<+.8ex>[l]^-G}
\end{displaymath}
\end{defi}

Notice that by definition, $\zeta$ is above $\eta$ and $\xi$ is above $\varepsilon$,
hence $(P,Q)$ is an ordinary map between adjunctions. Dually, we have the notion of \emph{general (lax) opfibred 
adjunction} in $\B{OpFib}_{(\lax)}$.

The following result establishes certain (co)cartesian properties of adjoints.

\begin{lem}\cite[4.5]{Winskel}\label{Winskellemma}
Right adjoints in the 2-category $\Cat^\B{2}$ preserve cartesian morphisms and dually left adjoints preserve cocartesian morphisms.
\end{lem}





Finally, in \cite[\S 3.2]{Measuringcomodule}, 
conditions under which a fibred 1-cell has an adjoint are investigated in detail, and that proves very useful
in determining enrichment relations in conjunction with \cref{actionenrich}.
Here we recall a main result providing a general lax opfibred adjunction, with regards to the applications
of \cref{sec:applications}.

\begin{thm}\label{totaladjointthm}
Suppose $(K,F):U\to V$ is an opfibred 1-cell and $F\dashv G$ is an adjunction with counit $\varepsilon$ between 
the bases of the opfibrations, as in
\begin{displaymath}
\xymatrix @C=.6in
{\ca{C}\ar[r]^-K\ar[d]_-U & \ca{D}\ar[d]^-V \\
\caa{X}\ar @<+.8ex>[r]^-F
\ar@{}[r]|-\bot
& \caa{Y}. \ar @<+.8ex>[l]^-G}
\end{displaymath}
If, for each $Y\in\caa{Y}$, the composite functor
$\ca{C}_{GY}\xrightarrow{K_{GY}}\ca{D}_{FGY}\xrightarrow{(\varepsilon_Y)_!}\ca{D}_Y$ between the fibres
has a right adjoint for each $Y\in\caa{Y}$, then $K$ has a right adjoint $R$ between the total categories
and $(K,F)\dashv(R,G)$ is an general oplax adjunction. 
\end{thm}

\section{Enriched fibrations}

This section's goal is to introduce a notion of an enriched fibration. It will do so in a way
that an adjusted version of \cref{actionenrich}, instead of providing an enrichment of an ordinary category
in a monoidal category, will give
an enrichment of an ordinary fibration in a \emph{monoidal} one.
The key idea is to shift all necessary structure (\cref{ex:Catpseudo}) from the context
of categories to fibrations,
moving from $(\B{Cat},\times,\ca{I})$ to the monoidal 2-category $(\B{Fib},\times,1_{\ca{I}})$.

First of all, a pseudomonoid (\cref{def:pseudomon}) in the 2-category of fibrations,
which will serve as the base of the enrichment, is a fibration
$T:\ca{V}\to\caa{W}$ equipped with a multiplication $m:T\times T\to T$
and unit $j:1_\ca{I}\to T$, along with 2-isomorphisms $a,\ell,r$ as in \cref{pseudomon2cells}.
More explicitly, the multiplication and unit are fibred 1-cells $m=(\otimes_\ca{V},\otimes_\caa{W})$ and $j=(I_\ca{V},I_\caa{W})$
\cref{eq:fibred1cell}, displayed as
\begin{equation}\label{multunitmonoidalfibr}
 \xymatrix @C=.6in
{\ca{V}\times\ca{V}\ar[r]^-{\otimes_\ca{V}}
\ar[d]_-{T\times T} & \ca{V}\ar[d]^-T \\
\caa{W}\times\caa{W}\ar[r]_-{\otimes_{\caa{W}}} & 
\caa{W}}\qquad\mathrm{and}\qquad
\xymatrix @C=.6in 
{\ca{I}\ar[r]^-{I_\ca{V}} \ar[d]_-{1} & 
\ca{V}\ar[d]^-{T} \\
\ca{I}\ar[r]_-{I_{\caa{W}}} & \caa{W}}
\end{equation}
where $\otimes_\ca{V}$ and $I_\ca{V}$ are cartesian, and invertible fibred 2-cells
$a=(a^\ca{V},a^\caa{W})$, $r=(r^\ca{V},r^\caa{W})$, $\ell=(\ell^\ca{V},\ell^\caa{W})$ \cref{eq:fibred2cell}, displayed as
\begin{displaymath}
 \xymatrix @C=1.2in @R=.6in
{\ca{V}\times\ca{V}\times\ca{V}\rtwocell^{\otimes(\otimes\times1)}
_{\otimes(1\times\otimes)}{\;a^\ca{V}}
\ar[d]_-{T\times T\times T} & \ca{V}\ar[d]^-T \\
\caa{W}\times\caa{W}\times\caa{W}
\rtwocell^{\otimes(\otimes\times1)}
_{\otimes(1\times \otimes)}{\;a^\caa{W}} & \caa{W}}
\end{displaymath}
\begin{displaymath}
  \xymatrix @C=.9in @R=.4in
{\ca{V}\times1\rtwocell^{\otimes(1\times I)}
_{\sim}{\;\;r^\ca{V}}
\ar[d]_-{T\times1} & \ca{V}\ar[d]^-T \\
\caa{W}\times1\rtwocell^{\otimes(1\times I)}
_{\sim}{\;\;r^\caa{W}} & \caa{W}}\qquad
\xymatrix @C=.8in @R=.4in
{1\times\ca{V}\rtwocell^{\otimes(I\times1)}
_{\sim}{\;\ell^\ca{V}}
\ar[d]_-{1\times T} & \ca{V}\ar[d]^-T \\
1\times\caa{W}\rtwocell^{\otimes(I\times1)}
_{\sim}{\;\;\ell^\caa{W}} & \caa{W}}
\end{displaymath}
where by definition $a^\ca{V},r^\ca{V},l^\ca{V}$ lie 
above $a^\caa{W},r^\caa{W},l^\caa{W}$.
The coherence axioms they satisfy turn out to give the usual axioms which make 
$(\ca{V},\otimes_\ca{V},I_\ca{V})$ and $(\caa{W},\otimes_\caa{W},I_\caa{W})$ 
into monoidal categories with the respective associativity, left and right unit constraints.

\begin{rmk}
The latter can be deduced also by the fact that the domain and codomain 2-functors
$dom,\;cod:\B{Fib}\subset\Cat^\B{2}\to\B{Cat}$ are in fact strict monoidal,
i.e. preserve the cartesian structure on the nose.
In other words, the equality of pasted diagrams
of 2-cells in $\B{Fib}$ breaks down into
equalities $\Cat$ for the two (ordinary) natural transformations they consist of.
\end{rmk}
Moreover, the strict commutativity of the diagrams (\ref{multunitmonoidalfibr})
implies that $T$ strictly preserves the tensor product and the unit object
between $\ca{V}$ and $\caa{W}$, \emph{i.e.}
\begin{displaymath}
 TA\otimes_\caa{W} TB=T(A\otimes_\ca{V} B),\quad I_\caa{W}=T(I_\ca{V}).
\end{displaymath}
Along with the conditions that $T(a^\ca{V})=a^\caa{W}$, 
$T(l^\ca{V})=l^\caa{W}$ and $T(r^\ca{V})=r^\caa{W}$, 
these data define a strict monoidal structure on $T$;
we obtain the following definition, which coincides
with \cite[12.1]{Framedbicats}.

\begin{defi}\label{def:monoidalfibration}
A \emph{monoidal fibration}
is a fibration $T:\ca{V}\to\caa{W}$ such that
\begin{enumerate}[(i)]
\item $\ca{V}$ and $\caa{W}$ are monoidal categories,
\item $T$ is a strict monoidal functor,
\item the tensor product $\otimes_\ca{V}$ of $\ca{V}$
preserves cartesian arrows.
\end{enumerate}
\end{defi}

If $\ca{V}$ and $\ca{W}$ are 
symmetric monoidal categories and $T$ is a symmetric strict monoidal
functor, we call $T$ a \emph{symmetric monoidal fibration}.
In a dual way, we can define a (\emph{symmetric})
\emph{monoidal opfibration} to be an opfibration
which is a (symmetric) strict monoidal functor, where the tensor
product of the total category preserves cocartesian arrows. Notice that a monoidal
opfibration is still a pseudomonoid (and not a pseudocomonoid), this time in $\OpFib$.
Finally, a \emph{monoidal bifibration} is one where the tensor product of the total category
preserves both cartesian and cocartesian liftings.

We now describe a pseudomodule for a pseudomonoid in $(\B{Fib},\times,1_\ca{I})$; in analogy to
\cref{actionenrich}, this will be the object which will eventually have the enriched structure. According to
\cref{def:pseudomod}, a \emph{pseudoaction} of a monoidal
fibration $T:\ca{V}\to\caa{W}$ on an ordinary fibration $P:\ca{A}\to\caa{X}$ is a fibred 1-cell
$\mu=(\mu^\ca{A},\mu^\caa{X}):T\times P\to P$
\begin{equation}\label{pseudoaction}
\xymatrix @C=.6in
{\ca{V}\times\ca{A}\ar[r]^-{\mu^\ca{A}}
\ar[d]_-{T\times P} & \ca{A}\ar[d]^-P \\
\caa{W}\times\caa{X}\ar[r]_-{\mu^{\caa{X}}} & 
\caa{X}}
\end{equation}
where $\mu^\ca{A}$ is cartesian, along with 2-isomorphisms $\chi,\nu$ as in \cref{pseudomod2cells}
in $\B{Fib}$. Explicitly, these are invertible
fibred 2-cells $\chi=(\chi^\ca{A},\chi^\caa{X})$,
$\nu=(\nu^\ca{A},\nu^\caa{X})$ represented by
\begin{displaymath}
\xymatrix @C=.4in @R=.003in
{& \ca{V}\times\ca{A}\ar@/^/[dr]^{\mu} & \\
\ca{V}\times\ca{V}\times\ca{A}\ar@/^/[ur]^-{M\times1}
\ar@/_/[dr]_-{1\times\mu}
\rrtwocell<\omit>{\;\;\chi^\ca{A}}
\ar[ddddd]_-{T\times T\times P} && \ca{A}\ar[ddddd]^-P \\
& \ca{V}\times\ca{A} \ar@/_/[ur]_-{\mu} & \\
\hole \\
\hole \\
& \caa{W}\times\caa{X} \ar@/^/[dr]^{\mu} & \\
\caa{W}\times\caa{W}\times\caa{X}
\ar@/^/[ur]^-{M\times1} \ar@/_/[dr]_-{1\times\mu}
\rrtwocell<\omit>{\;\;\chi^\caa{X}} && \caa{X}\\
& \caa{W}\times\caa{X} \ar@/_/[ur]_-{\mu} &}
\qquad
\xymatrix @C=.4in @R=.003in
{& \ca{V}\times\ca{A}\ar@/^/[dr]^{\mu} & \\
1\times\ca{A}\ar@/^/[ur]^-{I\times1}
\ar@/_3ex/[rr]_-{\sim}
\rrtwocell<\omit>{\;\;\nu^\ca{A}}
\ar[ddddd]_-{1\times P} && \ca{A}\ar[ddddd]^-P \\
\hole \\
\hole \\
\hole \\
& \caa{W}\times\caa{X} \ar@/^/[dr]^{\mu} & \\
1\times\caa{X}\ar@/^/[ur]^-{I\times1} 
\ar@/_3ex/[rr]_-{\sim}
\rrtwocell<\omit>{\;\;\nu^\caa{X}} && \caa{X}}
\end{displaymath}
where $\chi^\ca{A},\nu^\ca{A}$ are above 
$\chi^\caa{X},\nu^\caa{X}$ with respect to the 
appropriate fibrations. These data are subject to 
certain axioms, which in fact again
split up in two sets of commutative diagrams
for the the two natural isomorphisms that $\chi$ and $\nu$ consist of; these coincide with the action 
of a monoidal category axioms (\cref{def:action}).

\begin{defi}\label{def:Trepresentation}
A \emph{$T$-representation} for a monoidal fibration $T\colon\ca{V}\to\caa{W}$
is a fibration $P\colon\ca{A}\to\caa{X}$ equipped with a $T$-pseudoaction $\mu=(\mu^\ca{A},\mu^\caa{X})$. 
This amounts to two actions 
\begin{gather*}
 \mu^\ca{A}=*:\ca{V}\times\ca{A}\longrightarrow\ca{A}\nonumber \\
\mu^\caa{X}=\diamond:\caa{W}\times\caa{X}\longrightarrow\caa{X}
\end{gather*}
of the monoidal categories $\ca{V}$, $\caa{W}$ on the 
categories $\ca{A}$ and $\caa{X}$ respectively, satisfying the commutativity of (\ref{pseudoaction})
where $\mu^\ca{A}$ preserves cartesian arrows, such that for all $X,Y\in\ca{V}$ and $A\in\ca{A}$ the following conditions hold:
\begin{equation}\label{twoactionscompatible}
P\chi^\ca{A}_{XYA}=\chi^\caa{X}_{(TX)(TY)(PA)},\quad P\nu^\ca{A}_A=\nu^\caa{X}_{PA}.
\end{equation}
\end{defi}

The compatibility conditions of the above definition are natural, since by (\ref{pseudoaction})
\begin{displaymath}
 P(X*A)=TX\diamond PA
\end{displaymath}
for any $X{\in}\ca{V}$, $A{\in}\ca{A}$, hence 
the isomorphisms
$\chi^\ca{A}_{XYA}:X*(Y*A)\cong(X\otimes_\ca{V} Y)*A$
in $\ca{A}$ lie above certain isomorphisms
\begin{equation}\label{eq:iso1}
 P\chi^\ca{A}_{XYA}:TX\diamond(TY\diamond PA)\xrightarrow{\;\sim\;}
(TX\otimes_\caa{W} TY)\diamond PA
\end{equation}
in $\caa{X}$, due to the strict monoidality of $T$.
Similarly, $\nu^\ca{A}_A:I*A\cong A$ is mapped to 
\begin{equation}\label{eq:iso2}
 P\nu^\ca{A}_A:I_\caa{X}\diamond PA\xrightarrow{\;\sim\;}PA
\end{equation}
since $P(I_\ca{V}*A)=T(I_\ca{V})\diamond PA=I_\caa{W}\diamond PA$. Thus \cref{twoactionscompatible}
demand that \cref{eq:iso1,eq:iso2} coincide with the components of 
$\chi^{\caa{X}}$ and $\nu^{\caa{X}}$, from the $\caa{W}$-representation
$\caa{X}$.

The last step in modifying \cref{actionenrich} to obtain a correspondence between representations
of a monoidal fibration and the desired enriched fibrations, is to introduce a notion of a parameterized adjunction in $\B{Fib}$.
For that, we first re-formulate the `adjunctions with a parameter' \cref{thm:parameterCat} in the context of $\B{Cat}^\B{2}$.

\begin{thm}\label{thm:parameterCat2}
 Suppose we have a morphism $(F,G)$ of two variables
in $\B{Cat}^\B{2}$, given by a commutative square
of categories and functors
\begin{equation}\label{FG}
 \xymatrix @C=.6in @R=.4in
{\ca{A}\times\ca{B}\ar[r]^-F
\ar[d]_-{H\times J} & \ca{C}\ar[d]^-K \\
\caa{X}\times\caa{Y}\ar[r]_-G & \caa{Z}.}
\end{equation}
Assume that, for every $B\in\ca{B}$ and $Y\in\caa{Y}$, 
there exist adjunctions $F(-,B)\dashv R(B,-)$ and
$G(-,Y)\dashv S(Y,-),$ such that $(F(-,B),G(-,JB))$ has a 
right adjoint $(R(B,-),S(JB,-))$ in $\B{Cat}^\B{2}$. This is represented by
\begin{equation}\label{adjunction[2,Cat]}
 \xymatrix @C=.9in @R=.5in
{\ca{A}\ar[d]_-H\ar@<+.8ex>[r]^-{F(-,B)}
\ar@{}[r]|-{\bot} &
\ca{C}\ar[d]^-K\ar@<+.8ex>[l]^-{R(B,-)} \\
\caa{X}\ar@<+.8ex>[r]^-{G(-,JB)}\ar@{}[r]|-{\bot} & 
\caa{Z}\ar@<+.8ex>[l]^-{S(JB,-)}}
\end{equation}
where $(H,K)$ is a map of adjunctions (both squares commute and $\varepsilon_K=K\varepsilon$, $H\eta=\eta_H$).
Then, there is a unique way to define a morphism
of two variables 
\begin{equation}\label{defRS}
 \xymatrix @C=.6in @R=.4in
{\ca{B}^\op\times\ca{C}\ar[r]^-R
\ar[d]_-{J^\op\times K} & \ca{A}\ar[d]^-H \\
\caa{Y}^\op\times\caa{Z}\ar[r]_-S & \caa{X}}
\end{equation}
in $\B{Cat}^2$,
for which $\ca{C}(F(A,B),C)\cong\ca{A}(A,R(B,C)), \caa{Z}(G(X,Y),Z)\cong\caa{X}(X,S(Y,Z))$
are natural in all three variables.
\end{thm}
\begin{proof}
The result clearly follows from ordinary parameterized adjunctions. 
The fact that $(R(B,-),S(JB,-))$ is
an arrow in $\B{Cat}^\B{2}$ for all $B$'s ensures that the 
diagram (\ref{defRS}) commutes on the second
variable, and also on 
the first variable on objects, since 
$HR(B,C)=S(JB,KC)$. On arrows, commutativity follows
from the 
unique way of defining $R(h,1)$ and 
$S(Jh,1)$ for any $h:B\to B'$ under these assumptions.
\end{proof}

We call $(R,S)$ the \emph{parameterized adjoint} of $(F,G)$ in $\B{Cat}^\B{2}$, written $(F,G)\padj(R,S)$.

\begin{rmk}
Although the notion of an adjunction can be internalized in
any bicategory, its parameterized version seems to be much more involved.
In any monoidal bicategory with duals, we could ask 
for 1-cells $A\cong A\otimes I\xrightarrow{1\times b}A\otimes B\xrightarrow{t}C$
to have adjoints $g_b\colon C\to A$, for every $b\colon I\to B$.
For the cartesian 2-monoidal case at least, with `category-like' objects
like in $\Fib$, the 2-categorical approach of \cite[Thm. 2.4]{Multivariableadjunctions}
clarifies things.
\end{rmk}

Restricting to fibrations, 
consider a morphism of two variables in $\B{Fib}\subset\B{Cat}^2$ \emph{i.e.} a fibred 1-cell $(F,G)$ as in (\ref{FG}) with $F$ cartesian,
with the property that (\ref{adjunction[2,Cat]}) is a general fibred adjunction as in \cref{generalfibredadjunction},
\emph{i.e.} the partial right adjoint $R(B,-)$ is also cartesian. 
Dually, in $\B{OpFib}$ 
we request both $F$ and $R(B,-)$ to be cocartesian. Notice that in both cases,
the parameterized adjoint of two variables $(R,S)$ can neither be a fibred nor an opfibred 1-cell `wholly',
since at \cref{defRS} the vertical $J^\op\times K$ is a product of a fibration with an opfibration, hence neither of the two.

If we lift the (co)cartesian requirements, we end up with the (op)lax version of these adjunctions.
Since those cases are the most relevant to our examples, we abuse notation as to call (op)fibred
parameterized adjunctions the (op)lax ones. Based on the remark that follows, this abuse is in fact only fractional.

\begin{rmk}
 There exists an interesting asymmetry with respect to the (co)ca\-rte\-sian\-ness requirement
 of the left/right partial adjoints, due to \cref{Winskellemma}.
 Since right adjoints always preserve cartesian arrows in $\B{Cat}^2$
 and dually left adjoints always preserve cocartesian ones, we can deduce that any fibred 1-cell $(F,G)$
 has a (right) fibred parameterized adjoint as long as it has a $\Cat^\B{2}$-parameterized adjoint.
 Dually, an opfibred 1-cell has a (left) opfibred parameterized adjoint as long as it has it in $\Cat^\B{2}$.
\end{rmk}

\begin{defi}\label{def:fibredpadjoint}
Suppose $H$, $K$ are fibrations. A \emph{fibred parameterized adjunction} is a parameterized adjunction $(F,G)\padj(R,S)$ in $\Cat^\B{2}$,
between two 1-cells
\begin{displaymath}
  \xymatrix @C=.6in @R=.4in
{\ca{A}\times\ca{B}\ar[r]^-F
\ar[d]_-{H\times J} & \ca{C}\ar[d]^-K \\
\caa{X}\times\caa{Y}\ar[r]_-G & \caa{Z}}\quad\quad
 \xymatrix @C=.6in @R=.4in
{\ca{B}^\op\times\ca{C}\ar[r]^-R
\ar[d]_-{J^\op\times K} & \ca{A}\ar[d]^-H \\
\caa{Y}^\op\times\caa{Z}\ar[r]_-S & \caa{X}}
\end{displaymath}
where $R(B,-)$ is by default cartesian. Dually, an \emph{opfibred parameterized adjunction}
is as above, where $F(-,B)$ is by default cocartesian.
\end{defi}

The proposed definition of an enriched fibration is justified by the subsequent \cref{thm:actionenrichedfibr}
which fulfills our initial goal, i.e. to generalize \cref{actionenrich} to the context of (op)fibrations.
In \cref{equivenrfibr} we give an equivalent formulation in terms of enriched functors.
The enriched hom-functor is defined as in \cref{enrichedhomfunctor}, writing $\ca{A}$ for both the enriched and the underlying category.

\begin{defi}[\emph{Enriched Fibration}]\label{def:enrichedfibration}
Suppose $T:\ca{V}\to\caa{W}$ is a monoidal fibration. A fibration $P:\ca{A}\to\caa{X}$
is \emph{enriched} in $T$ when the following
conditions are satisfied:
\begin{itemize}
 \item the total category $\ca{A}$ is enriched in the total monoidal category $\ca{V}$ and the 
base category $\caa{X}$ is enriched in the base monoidal category $\caa{W}$,
in such a way that the following commutes: 
\begin{equation}\label{enrichedfibrationhom}
\xymatrix @C=.8in @R=.5in
{\ca{A}^\op\times\ca{A}\ar[r]^-{\ca{A}(-,-)}
\ar[d]_-{P^\op\times P} & \ca{V}\ar[d]^-T \\
\caa{X}^\op\times\caa{X}\ar[r]_-{\caa{X}(-,-)} &
\caa{W}}
\end{equation}
\item the composition law and the identities of the 
enrichments are compatible, in the 
sense that
\begin{align}\label{compositionidentitiesabove}
TM^{\ca{A}}_{A,B,C}&=M^{\caa{X}}_{PA,PB,PC} \\
Tj^\ca{A}_A&=j^{\caa{X}}_{PA}\notag
\end{align}
\end{itemize}
\end{defi}

The compatibilities \cref{compositionidentitiesabove} only state that the
composition and identities
\begin{displaymath}
M^{\ca{A}}_{A,B,C}:\ca{A}(B,C)\otimes_\ca{V}\ca{A}(A,B)\to\ca{A}(A,C), \quad
j^\ca{A}_A:I_\ca{V}\to\ca{A}(A,A)
\end{displaymath}
of the $\ca{V}$-enriched $\ca{A}$ are mapped, under $T$, exactly to those of the $\ca{W}$-enriched $\caa{X}$:
\begin{displaymath}
 M^{\caa{X}}_{\scriptscriptstyle{PA,PB,PC}}:\caa{X}(PB,PC)\otimes_\caa{W}\caa{X}(PA,PB)\to\caa{X}(PA,PC),\;\;
j^{\caa{X}}_{\scriptscriptstyle PA}:I_\caa{W}\to\caa{X}(PA,PA)
\end{displaymath}
where the domains and codomains already coincide by 
strict monoidality of $T$ and the commutativity of (\ref{enrichedfibrationhom}).

For the above definition, it could be argued that some sort of cartesian condition for
the enriched hom-functor $\ca{A}(-,-)$ should be asked; notice however that for $P$
a fibration, the product $P^\op\times P$ has neither a fibration not an opfibration structure.
If we required that the partial functor $\ca{A}(A,-)\colon\ca{A}\to\ca{V}$
is cartesian, all results below would still be valid with minor adjustments.
Since the examples (\cref{sec:applications}) so far do not seem to satisfy this extra condition, for the moment
we adhere to this more general definition.

\begin{rmk}[\emph{Enriched fibrations as enriched functors}]\label{equivenrfibr}
For a $T$-enriched fibration $P$ as above, $T$'s strict monoidal structure induces a 2-functor $\wt{T}\colon\VCat\to\WCat$
by \cref{changeofbase}. Hence we can make the $\ca{V}$-category
$\ca{A}$ into a $\caa{W}$-category $\wt{T}\ca{A}$, 
with the same set of objects $\ob\ca{A}$ and hom-objects $T\ca{A}(A,B)=\caa{X}(PA,PB)$.
Then $P\colon\ca{A}\to\caa{X}$ can be verified to have the structure of a $\caa{W}$-enriched functor between the 
$\caa{W}$-categories $\wt{T}\ca{A}$ and $\caa{X}$, with hom-objects mapping
$T\ca{A}(A,B)\xrightarrow{\;=\;}\caa{X}(PA,PB)$. The compatibility
with the composition and the identities is ensured by (\ref{compositionidentitiesabove}).

From this perspective, the definition of a $(T\colon\ca{V}\to\caa{W})$-enriched fibration between a $\ca{V}$-category $\ca{A}$
and a $\caa{W}$-category $\caa{X}$ could be reformulated as
a strictly fully faithful $\caa{W}$-enriched functor $P\colon\wt{T}\ca{A}\to\caa{X}$, whose
underlying ordinary functor $P_0\colon\ca{A}_0\to\caa{X}_0$ is a fibration (the commutativity of \cref{enrichedfibrationhom} follows).
\end{rmk}


Dually, we have the notion of an \emph{enriched opfibration}, as well as the following
combined version.

\begin{defi}\label{def:fibrenrichedinopfibr}
Suppose that $T:\ca{V}\to\caa{W}$ is a symmetric monoidal opfibration.
We say that a fibration $P:\ca{A}\to\caa{X}$ is \emph{enriched} in $T$
if the opfibration $P^\op:\ca{A}^\op\to\caa{X}^\op$ is 
an enriched $T$-opfibration. 
\end{defi}

Finally, we prove that to give a fibration with an action $(*,\diamond)$ of a monoidal fibration $T$ (\cref{def:Trepresentation}) with a 
fibred parameterized adjoint (\cref{def:fibredpadjoint}), is to give a $T$-enriched fibration (\cref{def:enrichedfibration}).

\begin{thm}\label{thm:actionenrichedfibr}
Suppose that $T:\ca{V}\to\caa{W}$ is a monoidal fibration,
which acts on an (ordinary) fibration $P:\ca{A}\to\caa{X}$
via the fibred 1-cell
\begin{displaymath}
 \xymatrix @C=.6in
{\ca{V}\times\ca{A}\ar[r]^-{*}\ar[d]_-{T\times P} & 
\ca{A}\ar[d]^-P \\
\caa{W}\times\caa{X}\ar[r]_-{\diamond} & \caa{X}.}
\end{displaymath}
If this action has a parameterized adjoint $(R,S)$ as in 
\begin{displaymath}
\xymatrix @C=.6in
{\ca{A}^\op\times\ca{A}\ar[r]^-{R}\ar[d]_-{P^\op\times P} & 
\ca{V}\ar[d]^-T \\
\caa{X}^\op\times\caa{X}\ar[r]_-{S} & \caa{W}} 
\end{displaymath}
we can enrich the fibration 
$P$ in the monoidal fibration $T$.
\end{thm}

\begin{proof}
Recall by \cref{def:Trepresentation} that the $T$-action in particular consists of two actions $*$ and $\diamond$
of the monoidal categories
$\ca{V}$ and $\caa{W}$ on the categories $\ca{A}$ and $\caa{X}$ respectively. Since $(*,\diamond)\padj(R,S)$,
by \cref{thm:parameterCat2} 
we have two ordinary adjunctions 
\begin{displaymath}
\xymatrix @C=.5in
{\ca{V}\ar@<+.8ex>[r]^-{-*A}
\ar@{}[r]|-\bot & \ca{A}\ar@<+.8ex>[l]^-{R(A,-)}}
\quad\mathrm{and}\quad
\xymatrix @C=.5in
{\caa{W}\ar@<+.8ex>[r]^-{-\diamond X}
\ar@{}[r]|-\bot & \caa{X}\ar@<+.8ex>[l]^-{S(X,-)}}
\end{displaymath}
for all $A\in\ca{A}$ and $X\in\caa{X}$. By \cref{actionenrich}, there exists a $\ca{V}$-category $\underline{\ca{A}}$ with underlying
category $\ca{A}$ and hom-objects $\underline{\ca{A}}(A,B)=R(A,B)$ and also
a $\caa{W}$-category $\underline{\caa{X}}$ with underlying category $\caa{X}$ 
and hom-objects $\underline{\caa{X}}(X,Y)=S(X,Y)$. Also, the enriched hom-functors
satisfy the required commutativity $TS(-,-)=R(P-,P-)$ by \cref{defRS}. 

Finally, we need to show that the composition and identity laws
of the enrichments are compatible as in (\ref{compositionidentitiesabove}), i.e. $TM^\ca{A}_{A,B,C}=M^\caa{X}_{PA,PB,PC}$ and
$Tj^\ca{A}_A=j^\caa{X}_{PA}$. For that, it is enough
to confirm that their adjuncts under $(-\diamond X)\dashv S(X,-)$ coincide.
The latter ones are explicitly given by \cref{compositionunderadjunction}
and \cref{identityunderadjunction}, i.e.
\begin{gather*}
\begin{tikzcd}[ampersand replacement = \&,column sep=.8in]
\left(S(PB,PC)\otimes S(PA,PB)\right)\diamond PA\ar[r,"\chi^{\caa{X}}"]\ar[ddr,dashed] \&
S(PB,PC)\diamond\left(S(PA,PB)\diamond PA\right)\ar[d,"1\diamond\varepsilon_{PB}"] \\
\& S(PB,PC)\diamond PB\ar[d,"\varepsilon_{PC}"] \\
\& PC
\end{tikzcd} \\
I\diamond PA\xrightarrow{\nu^\caa{X}}PA
\end{gather*}
For the former ones, since $(P,T)$ is a map of adjunctions
\begin{displaymath}
 \xymatrix @C=.9in @R=.5in
{\ca{V}\ar[d]_-T\ar@<+.8ex>[r]^-{-*A}
\ar@{}[r]|-{\bot} &
\ca{A}\ar[d]^-P\ar@<+.8ex>[l]^-{R(A,-)} \\
\caa{W}\ar@<+.8ex>[r]^-{-\diamond PA}\ar@{}[r]|-{\bot} & 
\caa{X},\ar@<+.8ex>[l]^-{S(PA,-)}}
\end{displaymath}
taking the images of $M^\ca{A}_{A,B,C}$ and $j^\ca{A}_A$ under $T$ and translating under the adjunction
$(-\diamond X)\dashv S(X,-)$ is the same as first translating under
$(-*A)\dashv R(A,-)$ and then applying $P$. That produces
\begin{gather*}
\begin{tikzcd}[ampersand replacement = \&,column sep=.8in]
P\left(R(B,C)\otimes R(A,B)\right)*A)\ar[r,"P\chi^{\ca{A}}"]\ar[ddr,dashed] \&
P(R(B,C)*R(A,B)*A)\ar[d,"P(1*\varepsilon_{B})"] \\
\& P(R(B,C)*B)\ar[d,"P(\varepsilon_{C})"] \\
\& PC
\end{tikzcd} \\
 P(I*A)\xrightarrow{P(\nu^\caa{A})}PA
\end{gather*}
Since $P\chi^{\ca{A}}=\chi^\caa{X}$
and $P\nu^\ca{A}=\nu^\caa{X}$ from \cref{twoactionscompatible}, and also $P\varepsilon=\varepsilon_P$ as a map
of adjunctions, the above composites coincide and the proof is complete.
\end{proof}

An important first example that should fit this setting of an action-induced enrichment
is that of a \emph{closed} monoidal fibration. 
Just like a monoidal closed category $\ca{V}$ is one where $\otimes\colon\ca{V}\times\ca{V}\to\ca{V}$ has a (right)
parameterized adjoint via $-\otimes X\dashv [X,-]$ for every object $X$, we can consider the following notion
based on \cref{def:fibredpadjoint}.

\begin{defi}\label{def:monoidalclosedfibr}
A monoidal fibration $T\colon\ca{V}\to\caa{W}$ is (right) \emph{closed} when its 
tensor product fibred 1-cell \cref{multunitmonoidalfibr}
\begin{displaymath}
\xymatrix @C=.6in
{\ca{V}\times\ca{V}\ar[r]^-{\otimes_\ca{V}}
\ar[d]_-{T\times T} & \ca{V}\ar[d]^-T \\
\caa{W}\times\caa{W}\ar[r]_-{\otimes_{\caa{W}}} & 
\caa{W}}
\end{displaymath}
has a parameterized adjoint
\begin{displaymath}
\xymatrix @C=.6in
{\ca{V}^\op\times\ca{V}\ar[r]^-{[-,-]_\ca{V}}
\ar[d]_-{T^\op\times T} & \ca{V}\ar[d]^-T \\
\caa{W}^\op\times\caa{W}\ar[r]_-{[-,-]_{\caa{W}}} & 
\caa{W}}
\end{displaymath}
Equivalently, by \cref{thm:parameterCat2}, $T$ is monoidal closed when
\begin{enumerate}[(i)]
 \item $\ca{V}$ and $\caa{W}$ are monoidal closed categories,
 \item $T$ is a strict closed functor, 
 \item $T\varepsilon=\varepsilon_T$ and $\eta_T=T\eta$ for the respective units and counits of the adjunctions.
\end{enumerate}
\end{defi}

Notice that by \cref{Winskellemma}, the right adjoint $[V,-]_\ca{V}$ between the total categories
automatically preserves cartesian liftings. On the other hand, for the dual notion of a
\emph{monoidal closed opfibration}, the right adjoint is not cocartesian by default.

\begin{rmk}
 In \cite[\S 13]{Framedbicats}, definitions of an \emph{internally closed} monoidal fibration over a cartesian monoidal base,
 as well as \emph{externally closed} monoidal fibration over an arbitrary monoidal base are given. These are equivalent to
 each other under certain hypotheses; none, however, guarantee that the total category is closed on its own right.
 The external definition gives some, but not all, conditions in terms of the fibres and the reindexing functors
 for a fibred adjoint to exist, in the spirit of results such as \cref{totaladjointthm}. 
\end{rmk}

Applying \cref{thm:actionenrichedfibr} we can deduce the enrichment of a monoidal closed
fibration in itself, analogously
to the ordinary case (\cref{rmk:VenrichedinV}). 

\begin{prop}\label{prop:enriched_closed_fib}
A monoidal closed fibration $T\colon\ca{V}\to\caa{W}$ is $T$-enriched. Dually, a monoidal
closed opfibration is enriched in itself.
\end{prop}

\begin{proof}
All clauses of \cref{def:Trepresentation} are satisfied, since
the functor $-\otimes_\ca{V}-$ is cartesian in both variables by \cref{def:monoidalfibration},
and also $Ta^\ca{V}_{XYZ}=a^\ca{W}_{TXTYTZ}$ and $T\ell^{\ca{V}}_X=\ell^{\caa{W}}_{TX}$ for
the respective associator and the left unitor since $T$ is a strict monoidal functor.
Therefore $(\otimes_\ca{V},\otimes_\caa{W})$ is indeed a $T$-action,
just like the regular representation of a monoidal category earlier. 
Since this action has a parameterized adjoint, by definition of a monoidal closed fibration, the result follows.
\end{proof}


Finally, there is a dual version to \cref{thm:actionenrichedfibr}, characterizing
the enrichment of an opfibration in a monoidal opfibration.

\begin{thm}\label{thm:actionenrichedopfibr}
Suppose that $T:\ca{V}\to\caa{W}$ is a monoidal opfibration,
which acts on an (ordinary) opfibration $U:\ca{B}\to\caa{Y}$
via the opfibred 1-cell
\begin{displaymath}
 \xymatrix @C=.6in
{\ca{V}\times\ca{B}\ar[r]^-{*}\ar[d]_-{T\times U} & 
\ca{B}\ar[d]^-U \\
\caa{W}\times\caa{Y}\ar[r]_-{\diamond} & \caa{Y}.}
\end{displaymath}
If this action has a parameterized adjoint $(R,S)$ as in
\begin{displaymath}
\xymatrix @C=.6in
{\ca{B}^\op\times\ca{B}\ar[r]^-{R}\ar[d]_-{U^\op\times U} & 
\ca{V}\ar[d]^-T \\
\caa{Y}^\op\times\caa{Y}\ar[r]_-{S} & \caa{W}} 
\end{displaymath}
we can enrich the opfibration $U$ in the monoidal opfibration $T$.
\end{thm}

\begin{rmk}
The asymmetry between cartesian and cocartesian functors with regards to fibred and opfibred adjunctions
is still apparent when comparing Theorems \ref{thm:actionenrichedfibr} and \ref{thm:actionenrichedopfibr}.
For the former, \cref{Winskellemma} ensures that the right parameterized adjoint will be, at least partially as $R(A,-)$,
cartesian; as a result, the whole parameterized adjunction lifts from $\Cat^\B{2}$ to $\Fib$.
On the other hand, for the latter dual theorem, the assumptions cannot ensure that the enriched hom $R$
will be partially cocartesian.
One reason for this discrepancy is that even if we change our setting from $\Fib_{(\opl)}$ to $\OpFib_{(\lax)}$, the enrichment
is given in both cases by the existence of a \emph{right} adjoint (and not of a left one in the dual setting).
\end{rmk}

\section{Applications}\label{sec:applications}

In this final chapter, we exhibit a few examples of the enriched fibration notion.
In these cases, Theorems \ref{thm:actionenrichedfibr} and \ref{thm:actionenrichedopfibr} seem to be the easiest way to deduce the enrichment,
due to the fact that the enrichments on the level of bases and total categories are themselves obtained
by the similarly-flavored \cref{actionenrich}. In what follows, we do not present all the relevant theory as it would take up many pages;
instead we provide the appropriate
references, in the hope that the interested reader will look for the details therein.

\subsection{(Co)modules over (co)monoids}\label{app2}

In the context of a locally presentable symmetric monoidal closed category $\ca{V}$, previous work \cite{Measuringcomonoid}
establishes an enrichment of the category of monoids $\Mon(\ca{V})$ in the symmetric monoidal
category of comonoids $\Comon(\ca{V})$, via \cref{actionenrich}. The action of comonoids on monoids is
induced by the internal hom of $\ca{V}$: for any coalgebra $C$ and algebra $B$, $[C,B]$ has always the structure
of an algebra via the convolution product. Its (right) parametrized adjoint $P\colon\Mon(\ca{V})^\op\times\Mon(\ca{V})
\to\Comon(\ca{V})$ which is the enriched hom-functor is called the \emph{Sweedler hom}, since the original notion of
a \emph{measuring coalgebra} $P(A,B)$ goes back to \cite{Sweedler}.

Furthermore, in \cite{Measuringcomodule} a similarly action-induced enrichment is established for the
\emph{global} category of modules in the symmetric monoidal global category of comodules, i.e. the category of all (co)modules over
any (co)monoid in $\ca{V}$. The action again comes from the internal hom of the monoidal category,
and its parametrized adjoint $Q\colon\Mod^\op\times\Mod\to\Comod$ maps an $A$-module $M$
and a $B$-comodule $N$ to their \emph{measuring comodule} $Q(M,N)$ \cite{Batchelor}. 
This parameterized adjunction is obtained itself using the theory of fibred adjunctions, since the functor $U\colon\Mod\to\Mon(\ca{V})$
which gives the `underlying' algebra is a fibration and $V\colon\Comod\to\Comon(\ca{V})$ is an opfibration.
Therefore the very enrichment on the level of the total categories
is accomplished via \cref{totaladjointthm}, producing a general (lax) opfibred adjunction
\begin{equation}\label{hugediag1}
\xymatrix @R=.65in @C=1in
{\Mod^\op \ar @/^/[r]^-{Q(-,N_B)}\ar@{}[r]|-{\top} \ar[d]_-{U^\op} &
\Comod \ar @/^/[l]^-{[-,N_B]^\mathrm{op}}\ar[d]^-V \\ 
\Mon(\ca{V})^\mathrm{op}\ar @/^/[r]^-{P(-,B)}
\ar@{}[r]|-{\top} & \Comon(\ca{V}) \ar @/^/[l]^-{[-,B]^\mathrm{op}}}
\end{equation}

To establish an enriched opfibration structure, we apply \cref{thm:actionenrichedopfibr}. First of all, 
$V\colon\Comod\to\Comon(\ca{V})$ can be shown to be a monoidal
opfibration, by definition of the monoidal product in $\Comod$. Moreover, since both actions on the level
of total and base categories are in fact the internal hom of $\ca{V}$ restricted to the appropriate subcategories,
compatibility \cref{twoactionscompatible} also follows. Finally, the top action is cartesian \cite[(20)]{Measuringcomodule} hence
the opfibration $U^\op$ is enriched in $V$, as in \cref{def:fibrenrichedinopfibr}.

\begin{prop}
Suppose $\ca{V}$ is a locally presentable symmetric
monoidal closed category. 
The fibration $\Mod\to\Mon(\ca{V})$ is enriched in the monoidal opfibration $\Comod\to\Comon(\ca{V})$.
\end{prop}

An example of such a monoidal category $\ca{V}$, which also motivated this whole development,
is the category of modules over a commutative ring, $\Mod_R$. Both enriching functors arise as adjoints
to the linear maps space functor $\Mod_R(-,-)$ restricted to the respective subcategories. In particular, for
two arbitrary modules $M$ and $N$ over $R$-algebras $A$ and $B$, the measuring comodule $Q(M,N)$
which provides the enrichment of modules in comodules has its coaction over the Sweedler's measuring $R$-coalgebra $P(A,B)$ which
provides the enrichment of algebras in coalgebras. Similarly, the comodule composition maps $Q(N,S)\otimes_R Q(M,N)\to Q(M,S)$
are above the coalgebra composition maps $P(B,C)\otimes_R P(A,B)\to P(A,B)$. 
This is a substantial step exhibiting the tight relations between these
dual-flavored, standard (op)fibrations.

Furthermore, the forgetful $\Comod\to\Comon(\ca{V})$ is in fact an example of a monoidal closed opfibration, \cref{def:monoidalclosedfibr}.
First of all, it is the case that the category of comonoids is monoidal closed, if $\ca{V}$ is locally presentable and 
symmetric monoidal closed, as was already proved in \cite[3.2]{MonComonBimon}.
In \cite[4.5]{Measuringcomodule}, it is shown in detail how the opfibred 1-cell
\begin{displaymath}
\xymatrix @C=.5in @R=.5in
{\Comod\times\Comod\ar[r]^-{(-\otimes-)} \ar[d] &
\Comod\ar[d] \\
\Comon(\ca{V})\times\Comon(\ca{V})\ar[r]_-{(-\otimes-)} & \Comon(\ca{V})}
\end{displaymath}
has a parameterized right adjoint, so the result follows by \cref{prop:enriched_closed_fib}.
\begin{prop}
 Suppose $\ca{V}$ is a locally presentable symmetric monoidal closed category. The monoidal opfibration
 $\Comod\to\Comon(\ca{V})$ is closed, therefore enriched in itself.
\end{prop}
Notice that this does not dualize for $\Mod\to\Mon(\ca{V})$, since in general the category of monoids
is not monoidal closed, e.g. rings or $R$-algebras. 

\subsection{Enriched (co)modules over enriched (co)categories}\label{app3}

The above study on enrichment relations between monoids and comonoids, as well as modules and comodules,
can be appropriately extended to their many-object generalizations, in the sense that
a monoid can be thought of as a one-object category.

For a detailed exposition of the notions and constructions that follow, see \cite[\S 7]{PhDChristina} or from a double
categorical perspective \cite{VCocats}.
Briefly, for a symmetric monoidal category with colimits preserved by $\otimes$, we can consider the category of $\ca{V}$-enriched categories
$\VCat$, whose objects are monads in the bicategory of \emph{enriched matrices} $\VMat$ \cite{Varthrenr}. In a dual way,
considering comonads therein, we can construct the category $\VCocat$ of \emph{enriched cocategories}, serving as
a many object generalization of comonoids in $\ca{V}$. A $\ca{V}$-cocategory $\ca{C}_X$ with set of objects $X$
comes equipped with cocomposition and coidentity arrows
\begin{displaymath}
\Delta_x,z\colon\ca{C}(x,z)\to\sum_{y\in X}\ca{C}(x,y)\otimes\ca{C}(y,z),\quad
\epsilon_x\colon\ca{C}(x,x)\to I
\end{displaymath}
in $\ca{V}$, satisfying coassociativity and counitality axioms.
Both categories $\VCat$ and $\VCocat$ are in fact fibred and opfibred, respectively, over
the category of sets, via the usual forgetful functors that give the set of objects of the (co)categories.
They also both live inside $\VGrph$, the category of enriched graphs, which is bifibred over
$\Set$. All these (op)fibrations are monoidal in the sense of \cref{def:monoidalfibration}:
for two $\ca{V}$-graphs (or categories, cocategories)
$G_X$ and $H_Y$, their tensor product is a graph $G\otimes H$ with set of objects $X\times Y$,
and $(G\otimes H)\left((x,y),(z,w)\right)=G(x,z)\otimes H(y,w)$.

When $\ca{V}$ is moreover monoidal closed with products and coproducts,
$\VGrph$ is monoidal closed: for two graphs $G_X$ and $H_Y$, their internal hom is the graph
$\Hom(G,H)$ with set of objects $Y^X$, given by the collection of $\ca{V}$-objects
$$\Hom(G,H)(k,s)=\prod_{x',x}[G(x',x),H(kx',sx)] \text{ for } k,s\in Y^X$$ 
By definition of these structures, the following diagram
\begin{displaymath}
\xymatrix @C=.9in @R=.5in
{\VGrph\ar[d]\ar@<+.8ex>[r]^-{-\otimes G_X}\ar@{}[r]|-{\bot} &
\VGrph\ar[d]\ar@<+.8ex>[l]^-{\Hom(G_X,-)} \\
\Set\ar@<+.8ex>[r]^-{-\times X}\ar@{}[r]|-{\bot} & 
\Set\ar@<+.8ex>[l]^-{(-)^X}}
\end{displaymath}
is a map of adjunctions, therefore all three clauses of \cref{def:monoidalclosedfibr} are satisfied.
\begin{prop}
 Suppose $\ca{V}$ is a symmetric monoidal closed category with products and coproducts. The bifibration
 $\VGrph\to\Set$ mapping a $\ca{V}$-graph to its set of objects is monoidal closed,
 therefore it is enriched in itself.
\end{prop} 

Similarly to how the internal hom of $\ca{V}$ was lifted to an action of comonoids on
monoids in \cref{app2}, the internal hom of $\VGrph$ induces an action
\begin{displaymath}
K:\xymatrix@R=.05in
{\ca{V}\text{-}\B{Cocat}^\op\times\ca{V}\text{-}\B{Cat}
\ar[r] & \ca{V}\text{-}\B{Cat} \\
\qquad(\;\ca{C}_X\;,\;\ca{B}_Y\;)\ar@{|->}[r] & 
\Hom(\ca{C},\ca{B})_{Y^X}}
\end{displaymath}
Its opposite has a parameterized adjoint, again by \cref{totaladjointthm},
\begin{displaymath}
T:\ca{V}\text{-}\B{Cat}^\op\times\ca{V}\text{-}\B{Cat}
\to\ca{V}\text{-}\B{Cocat}
\end{displaymath}
called the \emph{generalized Sweedler hom}. We get the following (lax) opfibred pa\-ra\-me\-te\-ri\-zed adjunction 
\begin{displaymath}
\xymatrix @R=.6in @C=1in 
{\ca{V}\textrm{-}\B{Cat}^\mathrm{op}\ar @/^/[r]^-{T(-,\ca{B}_Y)}
\ar[d]\ar@{}[r]|-{\top} &
\ca{V}\textrm{-}\B{Cocat}\ar @/^/[l]^-{K(-,\ca{B}_Y)^\mathrm{op}}\ar[d] \\ 
\B{Set}^\mathrm{op}\ar @/^/[r]^-{Y^{(-)}}\ar@{}[r]|-{\top} &
\B{Set} \ar @/^/[l]^-{{Y^{(-)}}^\mathrm{op}}}
\end{displaymath}
thus \cref{thm:actionenrichedopfibr} applies again, see \cite[4.38]{VCocats}.

\begin{prop}
Suppose $\ca{V}$ is a locally presentable, monoidal closed category.
The fibration $\VCat\to\Set$ is enriched in the monoidal opfibration $\VCocat\to\Set$,
where both functors send the enriched structure to its set of objects.
\end{prop}

Finally, we can consider many object generalizations of modules and comodules,
namely $\ca{V}$-modules for $\ca{V}$-categories
and $\ca{V}$-comodules for $\ca{V}$-cocategories, see \cite[7.6]{PhDChristina}. The former are quite standard:
an $\ca{A}_X$-module $\Psi$ can also be thought as a $\ca{V}$-profunctor
$\Psi\colon\ca{I}\tickar\ca{A}$ for $\ca{I}$ the unit category.
Objects are the same as $\ca{A}$, and hom-objects are $\Psi(x)\in\ca{V}$ equipped with
$(\sum_{x,x'})\ca{A}(x,x')\otimes\Psi(x')\to\Psi(x')$ satisfying appropriate axioms.
The notion of comodules is dual, and these form global categories much like before, $\VMod$ and $\VComod$.
The internal hom of enriched graphs further restricts to these categories, 
giving an action of $\VComod$ on $\VMod$ via a functor
\begin{displaymath}
\overline{K}:\xymatrix @R=.05in
{\ca{V}\text{-}\Comod^\op\times\ca{V}\text{-}\Mod
\ar[r] & \ca{V}\text{-}\Mod\qquad\\
\qquad(\;\Phi_\ca{C}\;,\;\Psi_\ca{B}\;)\ar@{|->}[r] & 
\Hom(\Phi,\Psi)_{\Hom(\ca{C},\ca{B})}}
\end{displaymath}
where $\Hom(\Phi,\Psi)(t)=\prod_{x}
[\Phi(x),\Psi(tx)]$. It has a parameterized adjoint
\begin{displaymath}
\overline{T}:\ca{V}\text{-}\Mod^\op\times\ca{V}\text{-}\Mod
\to\ca{V}\text{-}\Comod
\end{displaymath}
by \cref{totaladjointthm} which once more heavily relies on the fact
that $\VMod$ is fibred over $\VCat$ and $\VComod$ is opfibred $\VCocat$,
and there exists an adjunction between the base categories:
\begin{equation}\label{hugediag2}
\xymatrix @R=.6in @C=1in 
{\ca{V}\text{-}\Mod^\op\ar@/^/[r]^-{\overline{T}(-,\Psi_\ca{B})}
\ar[d] \ar@{}[r]|-{\top} &
\ca{V}\text{-}\Comod \ar@/^/[l]^-{\overline{K}(-,\Psi_\ca{B})^\mathrm{op}}\ar[d] \\ 
\ca{V}\textrm{-}\B{Cat}^\mathrm{op}\ar @/^/[r]^-{T(-,\ca{B}_Y)}\ar@{}[r]|-{\top} &
\ca{V}\textrm{-}\B{Cocat}\ar @/^/[l]^-{K(-,\ca{B}_Y)^\mathrm{op}}}
\end{equation}

\begin{prop}
If $\ca{V}$ is a locally presentable symmetric monoidal closed categor,
the fibration $\VMod\to\VCat$ is enriched in the opfibration $\VComod\to\VCocat$.
\end{prop}

\begin{proof}
\cref{thm:actionenrichedopfibr} applies, to first establish the enrichment of the opfibration
$\VMod^\op\to\VCat^\op$. First of all, $\VComod\to\VCocat$ is a monoidal opfibration by definition of the respective products
and cartesianness of $\otimes_{\VComod}$, \cite[7.7.6]{PhDChristina}. The commutative square of categories and functors
\begin{displaymath}
\xymatrix @C=.6in
{\VComod\times\VMod^\op\ar[r]^-{\overline{K}^\op}\ar[d] & \VMod^\op\ar[d] \\
\VCocat\times\VCat^\op\ar[r]_-{K^\op} & \VCat^\op}
\end{displaymath}
constitutes an opfibred action, since both $K$ and $\overline{K}$ are actions, $\overline{K}$ preserves cartesian arrows
by \cite[7.7.3]{PhDChristina} and the action axioms are the one above each other as per their definitions. Finally,
this opfibred 1-cell has an oplax opfibred parameterized adjoint by \cite[7.7.5]{PhDChristina}, and the proof is complete.
\end{proof}

\subsection{Comparison with existing notions}\label{sec:discussion}

The above examples were the ones that motivated the proposed enriched fibration notion -- although more should be identified
in future work. In this final section, we would like to discuss why other existing approaches
were not applicable, due to the nature of these cases.

Recall that assuming the axiom of choice, one can construct an equivalence between fibrations $\ca{A}\to\caa{X}$
and indexed categories, i.e. pseudofunctors $\caa{X}^\op\to\Cat$ via the classic \emph{Grothendieck construction}
\cite{Grothendieckcategoriesfibrees}. More recently \cite{Framedbicats,MonoidalGr} this correspondence
has been lifted between the respective monoidal structures; we believe that a (global) enriched version
of the Grothendieck construction in the future, which in a fibrewise sense appears in \cite{BeardsleyWong}, will shed more light to the tight
connections between our enriched fibration notion and the ones that follow. For the moment, we only sketch some of the main
relevant theory and differences.
 
M. Bunge in \cite{Bunge} first introduced the notion of an $\caa{S}$-\emph{indexed} $V$-\emph{category}, for $\caa{S}$ an
elementary topos and $V$ an $\caa{S}$-indexed \emph{monoidal} category $V\colon\caa{S}^\op\to\MonCat$. The goal of this work
was to provide a general context in order to compare as well as clarify certain misconceptions
regarding different completions on 2-categories, such as the Karoubi, Grothendieck, Cauchy and Stack completion.

Independently, M. Shulman in \cite{Enrichedindexedcats} also develops a theory of \emph{enriched indexed categories} 
over base categories $\caa{S}$ with finite products. The motivation in that paper was to capture and study `mixed' fibred,
indexed and internal structures in various contexts, such as Parameterized and Equivariant Homotopy Theory, abelian sheaves and many more.

Briefly, for $\caa{S}$ cartesian monoidal, take $V$ to be an $\caa{S}$-indexed monoidal category, equivalently viewed as a monoidal fibration
$\int V\colon\ca{V}\to\caa{S}$. A $V$-enriched indexed category $A$ is simultaneously indexed (or fibred) over the same $\caa{S}$
and also `fibrewise' enriched in $\ca{V}$: every category (or fiber) $A(s)$ for $s\in\caa{S}$ is $V(s)$-enriched,
and the reindexing functors are fully faithful enriched under the appropriate change of base. 
Although this formulation employs the same notion of a monoidal fibration (\cref{def:monoidalfibration})
as the base of the enrichment, there are some crucial differences resulting in two separate definitions,
\cite[2.4]{Bunge} - \cite[4.1]{Enrichedindexedcats} and \cref{def:enrichedfibration}.

First of all, Bunge's and Shulman's approach only concerns enrichment in fibrations over monoidal categories
whose tensor product is the cartesian product. This is fundamental for the development and definitions,
and not a special case of something more general; of course this was relevant to their examples at hand.
On the contrary, for our examples this is evidently not the case:
in \cref{hugediag1,hugediag2} the base monoidal categories of the monoidal fibrations,
$\Comon(\ca{V})$ and $\ca{V}$-$\B{Cocat}$, are non-cartesian.

Moreover, the notion of an enriched indexed category roughly
expressed in the fibred world, essentially refers to a fibration `enriched' in another fibration over the same base,
approximately depicted as
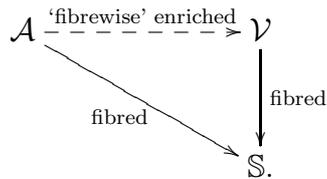
\begin{displaymath}
\xymatrix @C=1in @R=.5in
{\ca{A}\ar@{-->}[r]^-{\textrm{`fibrewise' enriched}} 
\ar[dr]_-{\mathrm{fibred}} & 
\ca{V} \ar[d]^-{\mathrm{fibred}} \\
& \caa{S}.} 
\end{displaymath}
In our examples, this fibrewise enrichment is certainly not the case: the fibre categories of our monoidal fibrations, like 
$\Comod_\ca{V}(C)$, do not even have a monoidal structure themselves in order to serve as enriching bases. 
Furthermore, even if in \cite[\S 7]{Enrichedindexedcats}
there is a short treatment of changing the indexed monoidal enriching base, and the development
in \cite{Bunge} is a special case of this via the identity functor on $\caa{S}$,
in our context the enriched fibration concept involves simultaneous
enrichments between both the total and the base categories of the two fibrations
as essential building blocks of the structure:
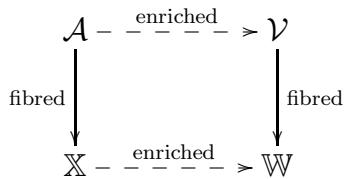
\begin{displaymath}
\xymatrix @C=.8in @R=.5in
{\ca{A}\ar@{-->}[r]^-{\mathrm{enriched}} 
\ar[d]_-{\mathrm{fibred}} & \ca{V} \ar[d]^-{\mathrm{fibred}} \\
\caa{X}\ar@{-->}[r]^-{\mathrm{enriched}} & \caa{W}} 
\end{displaymath}

In conclusion, even if there are strong conceptual similarities between the two definitions of an enriched fibration and
indexed $\ca{V}$-category, our definition does not seem to even restrict in a straightforward way 
to the case of fibrations over the same base, since  
the monoidal category $\caa{W}$ is not 
in principle enriched over itself, nor via some sort of an identity or projection functor.
As mentioned earlier, future work would aim to clarify how these two theories compare in more detail and depth.
What is admittedly striking though is that several different goals and motivations have separately led
to the need for a theory that combines fibred structure over a base topos or (cartesian) monoidal category,
and enriched structure.

\bibliographystyle{alpha}
\bibliography{myreferences}

\end{document}